\documentclass[reqno,11pt]{article}

\RequirePackage[OT1]{fontenc}
\RequirePackage{amsthm, amsmath, amssymb, amsfonts, natbib, xspace, graphicx, enumerate,multirow}

\usepackage{color,xr-hyper}

\usepackage{epsfig,amsmath,amssymb,euscript,amsfonts,mathrsfs}
\usepackage{colortbl,color,graphics,graphicx,epsfig,multirow,float}
\usepackage{framed}
\usepackage{dsfont}
\usepackage{bbm}
\usepackage{changepage}
\usepackage{nicefrac}
\usepackage{indentfirst} 
\usepackage{comment}
\usepackage{booktabs}
\usepackage{natbib}
\usepackage{arydshln}

\usepackage[top=1in, bottom=1in, left=1in, right=1in]{geometry}

\allowdisplaybreaks
\newtheorem{theorem}{Theorem}[section]

\newtheorem{assumption}{Assumption}[section]

\newtheorem{lem}{Lemma}[section]

\theoremstyle{plain}

\theoremstyle{remark}

\newtheoremstyle{assm} 
    {\topsep}                    
    {\topsep}                    
    {}                               
    {}                           
    {\scshape}                   
    {.}                          
    {.5em}                       
    {}  
\theoremstyle{assm}

\theoremstyle{definition}

\newcommand{\ignore}[1]{}

\begin{document}

\title{\vspace*{-2cm} Inference for Multiple Change-points in Linear and Non-linear Time Series Models}

\author{
Wai Leong Ng\thanks{Department of Statistics, The Chinese University of Hong Kong, Shatin, N.T., Hong Kong. Email: {\ttfamily s1009613922@sta.cuhk.edu.hk}}\\
The Chinese University of Hong Kong
\and
Shenyi Pan\thanks{Department of Statistics, The University of British Columbia, Vancouver, B.C., Canada. Email: {\ttfamily shenyi.pan@stat.ubc.ca}} \\
The University of British Columbia
\and
Chun Yip Yau\thanks{Department of Statistics, The Chinese University of Hong Kong, Shatin, N.T., Hong Kong.  Email: {\ttfamily cyyau@sta.cuhk.edu.hk}. 
Supported in part by HKSAR-RGC Grants CUHK405012, 405113 and 14601015} \\
The Chinese University of Hong Kong
\and 
}
\maketitle

\begin{abstract}
In this paper we develop a generalized likelihood ratio scan method (GLRSM) 
for multiple change-points inference in piecewise stationary time series, which estimates the number and positions of change-points and provides a confidence interval for each change-point.
The computational complexity of using GLRSM for multiple change-points detection is as low as $O(n(\log n)^3)$ for a series of length $n$. Consistency of the estimated numbers and positions of the change-points is established. 
Extensive simulation studies are provided to demonstrate the effectiveness of the proposed methodology under different scenarios. 
	
\end{abstract}

\section{Introduction}\label{sec:intro}

Change-point detection has been recognized as an important issue in econometrics, biology, and engineering for decades. 
Extensive literature has explored the detection of change-points in time series models. \cite{picard1985testing} first studied the maximum likelihood estimator (MLE) of change-points in autoregressive (AR) models. \cite{bai1998testing} extends Picard's method for change-points in multivariate AR and cointegrating models. 
Recently, multiple change-points problems have attracted growing attention. For example, 
\cite{ombao2001automatic} divides a time series into dyadic segments by minimizing an objective function and \cite{braun2000multiple} employs quasi-likelihood to detect multiple change-points in DNA sequences. 
Depending on whether the entire data set is considered as a whole or each data point is processed sequentially, change-point detection problems can be classified as off-line or on-line respectively. In this paper we focus on off-line change-point detection. 

Under the off-line setting, when the number of change-points is unknown, a common approach is to search for the set of change-points that minimizes certain objective functions, for example, the least-squares criterion for a change in the mean of a series (see \cite{yao1989least}, \cite{lavielle2000least}), 
and minimum description length criterion for changes in model parameters (\cite{davis2006structural}). 
However, such optimization procedures usually come with high computational costs because the number of possible combinations of change-points grows exponentially with sample size. 
Some attempts to reduce the computational burden include the genetic algorithm in \cite{davis2006structural, davis2008break} and the pruned exact linear time (PELT) method in \cite{killick2012optimal}. 
Nevertheless, the genetic algorithm involves many tuning parameters, while the computational cost of the PELT is close to  $O(n^2)$ if the number of change-points does not increase linearly with the length of the time series $n$. \cite{yau2015inference} proposed the likelihood ratio scan method (LRSM) to identify multiple change-points for piecewise stationary AR models. The computation can be performed with order $O(n\log n)$, which is lower than the standard order of $O(n^2)$.

In practice, merely reporting the positions of the structural changes in the series does not provide complete information. A more informative description of change-points could be provided by confidence intervals (CI). To obtain the CIs, the asymptotic distribution of the change-point estimator is required. 
There is an extensive literature concerning limiting distribution of change-point estimators for independence sequences. \cite{hinkley1970inference} investigated maximum likelihood estimator of change-points in a sequence of i.i.d. random variables and proved that the estimated change-point converges in distribution to the location of the maxima of a double-sided random walk. \cite{bhattacharya1976minimum} derived the limiting distribution of change-point estimators for independent errors under local changes. \cite{yao1987approximating} showed that Hinkley's (1970) limiting distribution can be approximated by distribution which can be analytically tackled. \cite{dumbgen1991asymptotic} investigated asymptotic behavior of non-parametric change-point estimators. \cite{bai1995least} studied a structural-changed
regression model and showed that the estimated change-point converges
in distribution to the location of the maxima of a double-sided random
walk. For a survey of various results, see \cite{csorgo1997limit} and \cite{antoch1999estimators}. 
For the time series case, \cite{antoch1997effect} derived the asymptotic distribution of the change-point estimator 
for linear process with a change in mean.  
Recently, the asymptotic distribution of maximum likelihood estimator for a change-point in the parameters of time series models 
was established by \cite{ling2016estimation}, showing that the estimated change-point converges weakly to the location of the maxima of a double-sided random walk. 

In this paper, we propose a generalized LRSM (GLRSM) for multiple change-point inference 
 in general time series models. Performing GLRSM involves three steps: first, a likelihood ratio scan statistic is used to obtain a possibly overestimated set of change-points; second, a model selection procedure is employed to give a set of consistent change-point estimates and, finally, a confidence interval is constructed for each estimated change-point. 
The GLRSM procedure thus provides a computationally efficient and automated procedure for change-point inference in time series. 

This paper is organized as follows. In Section \ref{sec:background}, we review the basic settings of the change-point problem. In Section \ref{sec:cpinference}, we develop the GLRSM on general piecewise stationary time series based on the confidence interval construction procedure. Simulation studies are reported in Section \ref{sec:simulation}. 
We conclude in Section \ref{sec:conclusion}. Technical proofs of the theorems and lemmas are provided in Section \ref{sec:proof}.

\section{Basic settings and assumptions}\label{sec:background}

In this section, we first review the basic settings of the change-point problem and  
the asymptotic distribution of the change-point estimator in \cite{ling2016estimation}. 

First we consider single change-point estimation in piecewise stationary processes. Assume that the time series $\{X_t: t=1,2,\dots\}$ is 
$\mathcal{F}_t$-measurable, strictly stationary, ergodic, and generated by
\begin{equation} \label{g-model}
X_t=g(\boldsymbol{\theta},\boldsymbol{X}_{t-1},\epsilon_t)\,, 
\end{equation}
where $\mathcal{F}_t$ is the $\sigma$-field generated by $\{\epsilon_t,\epsilon_{t-1},\dots\}$, $\boldsymbol{X}_t=(X_t,\dots,X_{t-q})$ or $\boldsymbol{X}_t=(X_t,X_{t-1},\dots)$, $\boldsymbol{\theta}$ is an unknown $p\times 1$ parameter vector, and $\{\epsilon_t\}$ are independently and identically distributed (i.i.d.) innovations. The structure of the time series $\{X_t\}$ is 
characterized by a measurable function $g$ and the parameter $\boldsymbol{\theta}$. 
We assume that the parameter space $\boldsymbol{\Theta}$ is a bounded compact subset of $\mathbb{R}^p$ and 
$g$ is continuous with respect to $\boldsymbol{\theta}$. 

We denote model \eqref{g-model} by $M(\boldsymbol{\theta}_0)$ when the true parameter is $\boldsymbol{\theta}=\boldsymbol{\theta}_0$. Let $\{X_1,\dots,X_n\}$ be a random sample realization consists of two independent processes 
\begin{equation}\label{eq:cp.model}
\{X_1,\dots,X_{\tau_1^0}\}\in M(\boldsymbol{\theta}_1^0) \textup{ and } \{X_{\tau_1^0+1},\dots,X_n\}\in M(\boldsymbol{\theta}_2^0)\,,
\end{equation}
where $\tau_1^0\in\{1,2,\dots,n-1\}$ is the change-point satisfying $\delta n<\tau_1^0<(1-\delta)n$ for some small $\delta>0$, and
$\boldsymbol{\theta}_1^0$, $\boldsymbol{\theta}_2^0\in\boldsymbol{\Theta}$ are unknown parameters in the two segments satisfying 
$\boldsymbol{\theta}_1^0\neq \boldsymbol{\theta}_2^0$.

Define the conditional likelihood function  $l_t(\boldsymbol{\theta})\equiv\log f_{\boldsymbol{\theta}}(X_t|\boldsymbol{X}_{t-1})$, where 
$f_{\boldsymbol{\theta}}$ is the conditional density of $X_t$ given previous observations.  
The likelihood for the change-point model \eqref{eq:cp.model} is 
\begin{eqnarray*}
L_n(\tau,\boldsymbol{\theta}_1,\boldsymbol{\theta}_2)=L_{1n}(\tau,\boldsymbol{\theta}_1)+L_{2n}(\tau,\boldsymbol{\theta}_2)\,,
\end{eqnarray*}
where $L_{1n}(\tau,\boldsymbol{\theta}_1)=\sum_{t=1}^{\tau}l_t(\boldsymbol{\theta}_1)$ and $L_{2n}(\tau,\boldsymbol{\theta}_2)=\sum_{t=\tau+1}^{n}l_t(\boldsymbol{\theta}_2)$ are the log-conditional likelihood functions for the first and second segment, respectively.  For a given $\tau$, let $\hat{\boldsymbol{\theta}}_{1n}(\tau)$ and $\hat{\boldsymbol{\theta}}_{2n}(\tau)$ be the maximizer of $L_{1n}(\tau,\boldsymbol{\theta}_1)$ and $L_{2n}(\tau,\boldsymbol{\theta}_2)$ on $\Theta$, respectively. The change-point $\tau$ is estimated by 
\begin{equation}\label{eq:cp.est}
\hat{\tau}=\mathop{\arg\max}_{1\leq\tau\leq n}L_n[\tau,\hat{\boldsymbol{\theta}}_{1n}(\tau),\hat{\boldsymbol{\theta}}_{2n}(\tau)]\,.
\end{equation}
The parameter estimates for the two segments are given by $\hat{\boldsymbol{\theta}}_{1}=\hat{\boldsymbol{\theta}}_{1n}(\hat{\tau})$ 
and $\hat{\boldsymbol{\theta}}_{2}=\hat{\boldsymbol{\theta}}_{2n}(\hat{\tau})$, respectively. Consistency of $\hat{\tau}$, 
$\hat{\boldsymbol{\theta}}_{1}$ and $\hat{\boldsymbol{\theta}}_{2}$ are established in \cite{ling2016estimation}.

\subsection{Asymptotic distribution of change-point estimators}\label{sec:cp.dist}

For the change-point model \eqref{eq:cp.model}, define the double-sided random walk
\begin{equation} \label{Wjdef}
W_{\tau}=\begin{cases}
\sum_{t=1}^{\tau}(l_t(\boldsymbol{\theta}_1^0)-l_t(\boldsymbol{\theta}_2^0))\,, &\tau>0\,,\\
0\,, &\tau=0\,,\\
\sum_{t=\tau}^{-1}(l_t(\boldsymbol{\theta}_2^0)-l_t(\boldsymbol{\theta}_1^0))\,, &\tau<0\,.
\end{cases}
\end{equation}
where $X_t \in M(\boldsymbol{\theta}_2^0)$ when $\tau>0$ and $X_t\in M(\boldsymbol{\theta}_1^0)$ when $\tau<0$. According to Theorem 2.2 (b) in \cite{ling2016estimation}, for fixed $\boldsymbol{\theta}_1^0$ and $\boldsymbol{\theta}_2^0$, we have that
\begin{equation}\label{tau.weak.con}
\hat{\tau}_1-\tau_1^0\overset{d}{\to}\mathop{\arg\max}_{\tau \in \mathbb{Z}}W_{\tau}\,.
\end{equation}

Note that the limit $W_{\tau}$ depends on the unknown parameters $\boldsymbol{\theta}_1^0$ and $\boldsymbol{\theta}_2^0$ and closed form expression for the 
distribution function is not available. Therefore, it is difficult to use \eqref{tau.weak.con} to construct confidence intervals in practice. 
Theorem 3.1 in \cite{ling2016estimation} derived an approximation for $W_{\tau}$ when the parameter change was small. Specifically, if $\boldsymbol{\theta}_1^0-\boldsymbol{\theta}_2^0 = O(1/\sqrt{n})$, then for any fixed $M$,
\begin{equation} \label{lingCI}
(\hat{d}'\hat{\Sigma} \hat{d})^2(\hat{d}'\hat{\Omega}\hat{d})^{-1} \left(\mathop{\arg\max}_{r \in [-M,M]}W_{\lfloor \hat{m}r \rfloor}\right)\overset{d}{\to}\mathop{\arg\max}_{r\in[-M,M]} \left\{B(r)-\frac{1}{2}|r| \right\} \,,
\end{equation}
where $B(r)$ is the standard Brownian motion in $\mathbb{R}$, $\hat{d}=\hat{\boldsymbol{\theta}}_1-\hat{\boldsymbol{\theta}}_2$, $\hat{\Sigma}=\frac{1}{2M}\sum_{t=-M}^{M}\frac{\partial^2l_t(\boldsymbol{\theta})}{\partial\boldsymbol{\theta}\partial\boldsymbol{\theta}'}|_{\hat{\boldsymbol{\theta}}_2}$, $\hat{\Omega}=\frac{1}{2M}\sum_{t=-M}^{M}(D_t(\hat{\boldsymbol{\theta}}_2)-\bar{D})(D_t(\hat{\boldsymbol{\theta}}_2)-\bar{D})'$, where $D_t(\boldsymbol{\theta})=\frac{\partial l_t(\boldsymbol{\theta})}{\partial\boldsymbol{\theta}}$ and $\bar{D}=\sum_{-M}^{M}\frac{D_t(\hat{\boldsymbol{\theta}}_2)}{2M}$, $\hat{m}=(\hat{d}'\hat{\Sigma} \hat{d})^{-2}(\hat{d}'\hat{\Omega}\hat{d})$ and $\lfloor a \rfloor$ is the largest integer not greater than $a$. 
Using \eqref{lingCI} for a sufficiently large $M$, \cite{ling2016estimation} suggests an approximate $100(1-\alpha)\%$ CI given by
\begin{equation}\label{ling14.CI}
CI=\big[\hat{\tau}_1-[\Delta F_{\alpha/2}]-1,\hat{\tau}_1+[\Delta F_{\alpha/2}]+1\big]\,,
\end{equation}
where $\Delta=(\hat{d}'\hat{\Omega}\hat{d})(\hat{d}'\hat{\Sigma} \hat{d})^{-2}$ and $F_{\alpha/2}$ is the $\frac{\alpha}{2}$-th quantile of the distribution $\mathop{\arg\max}_{r\in\mathbb{R}} \{B(r)-\frac{1}{2}|r|\}$.

\section{Estimation of multiple change-points using the Likelihood Ratio Scan Method}\label{sec:cpinference}

In this section, we propose an automated procedure for estimating and constructing confidence intervals for multiple change-points. This procedure generalizes the 
Likelihood Ratio Scan Method (LSRM) in \cite{yau2015inference} from piecewise stationary AR models to general time series models. 
First, we begin with some basic settings and assumptions.

\subsection{Basic settings and assumptions}

We assume that the observations $\{X_t\}_{t=1,\dots,n}$ can be partitioned into $m+1$ stationary processes. 
For $j=1,\dots,m$, the $j$-th segment $\mathbf{X}_{j}=\{X_{\tau_{j-1}+1},\dots,X_{\tau_j}\}$ follows model $M(\boldsymbol{\theta}_{j})$, 
and $\tau_j$ is the $j$-th change-point where the $j$-th segment of the process changes to the $(j+1)$-th segment. 
Set $\tau_0\triangleq 0$ and $\tau_{m+1}\triangleq n$. 
To develop asymptotic theory in change-point analysis, the length of each segment must increase as the sample size increases. Let $\mathcal{J}=(\tau_1,\dots,\tau_m)$ be the set of change-points. 
Define $\lambda_j$ to be the relative position of the $j$-th change-point satisfying $\tau_j=[\lambda_j n]$, $j=0,\dots,m+1$. 
We assume that $\min_{j=0,\dots,m}(\lambda_{j+1}-\lambda_j)>\epsilon_\lambda$ for some $\epsilon_\lambda>0$. 

\subsection{Multiple change-points detection using scan statistics}
In this subsection we state the three steps of the generalized LRSM (GLRSM) method for multiple change-points inference. 
Discussions and asymptotic properties of GLRSM are given in Section \ref{sec:asy.prop}.

\begin{itemize}
\item[]{\bf \large First step: Obtain potential change-points using scan statistics}  \\
For $t=h,\ldots,n-h$, define the scanning window at $t$ and the corresponding observations as
\begin{equation*}
W_t(h)=\{t-h+1,\dots,t+h\} \textup{ \ \ and \ \ } X_{W_t(h)}=(X_{t-h+1},\dots,X_{t+h})\,,
\end{equation*}
respectively, where $h$ is called the window radius. To establish asymptotic theory, we assume that $h=h(n)$ depends on the sample size $n$. 
Define the likelihood ratio scan statistic for the scanning window $W_t(h)$ by
\begin{equation*}
S_h(t)=\frac{1}{h}L_{1h}(t,\hat{\boldsymbol{\theta}}_1)+\frac{1}{h}L_{2h}(t,\hat{\boldsymbol{\theta}}_2)-\frac{1}{h}L_{\cdot h}(t,\hat{\boldsymbol{\theta}})\,,
\end{equation*}
where $L_{1h}(t,\hat{\boldsymbol{\theta}}_1)$, $L_{2h}(t,\hat{\boldsymbol{\theta}}_2)$, and $L_{\cdot h}(t,\hat{\boldsymbol{\theta}})$ are the quasi-likelihoods 
formed by the observations $\{X_s\}_{s=t-h+1,\dots,t}$, $\{X_s\}_{s=t+1,\dots,t+h}$, and $\{X_s\}_{W_h(t)}$, evaluated at the maximum likelihood estimators (MLE) $\hat{\boldsymbol{\theta}}_1$, $\hat{\boldsymbol{\theta}}_2$, and $\hat{\boldsymbol{\theta}}$, respectively. To be specific, 
given a sample $\boldsymbol{z}=\{z_1,\dots,z_n\}$, the quasi-likelihood is defined as 
\begin{equation*}
\boldsymbol{L}(\boldsymbol{\theta})=\sum_{t=1}^{n}l_t(\boldsymbol{\theta})\equiv\sum_{t=1}^{n}\log f_{\boldsymbol{\theta}}(z_t|z_{t-1},z_{t-2},\ldots )\,,
\end{equation*}
where $f_{\boldsymbol{\theta}}(z_t|z_{t-1},z_{t-2},\dots)$ is the conditional density of $z_t$ given previous observations and $z_s=0$ for $s\leq 0$. 

Next, scan the observed time series by using $S_h(t)$ to yield a sequence of likelihood ratio scan statistics $(S_h(h), S_h(h+1),\dots,S_h(n-h))$. 
By construction, if $t$ is a change-point, then $S_h(t)$ tends to be large. If $h$ is chosen such that $2h<n\epsilon_\lambda$, then at most one change-point exists inside each scanning window. Therefore, we can obtain a set of potential change-points from the local maximizers of $S_h(\cdot)$. 
Define the local change-point estimates as follows:
\begin{equation*}
\hat{\mathcal{J}}^{(1)}=\bigg\{m\in\{h,h+1,\dots,n-h\}:S_h(m)=\max_{t\in(m-h,m+h]}S_h(t)\bigg\}\,,
\end{equation*}
where $S_h(t)\triangleq 0$ for $t<h$ and $t>n-h$. That is, $m$ is a local change-point estimator if $S_h(m)$ is the maximum over the window $[m-h+1,m+h]$ centering at $m$. We denote the number of elements in $\hat{\mathcal{J}}^{(1)}$ by $\hat{m}^{(1)}$. 
To further improve the computation efficiency, in practice we may restrict the size of $\hat{\mathcal{J}}^{(1)}$ by keeping 
a pre-specified number of change-points with the greatest $S_{h}(m)$; see \cite{yau2015inference}.

\item[]{\bf \large Second step: Consistent estimation by model selection approach} \\ 

The set of potential change-points $\hat{\mathcal{J}}^{(1)}$ obtained from the first step usually overestimates the true set of change-points. 
To detect the true change-points, we select the best subset from $\hat{\mathcal{J}}^{(1)}$ based on certain information criteria. 
We adopt the minimum description length (MDL) criterion (see \cite{davis2006structural, davis2008break}) which is found to give promising 
performance in a number of empirical studies. Given a set of change-points $\mathcal{J}=(\tau_1,\ldots,\tau_m)$, the MDL criterion is defined as
\begin{equation*}
\textup{MDL}(m,\mathcal{J})=\log(m)+(m+1)\log(n)+\sum_{j=1}^{m+1}\sum_{k=1}^{c_j}\log(\zeta_{j,k})+\sum_{j=1}^{m+1}\frac{d_j}{2}\log(n_j)
-\sum_{j=1}^{m+1}L_j(\hat{\boldsymbol{\theta}}_j;\mathbf{X}_{j})\,,
\end{equation*}
where $L_j(\hat{\boldsymbol{\theta}}_j;\mathbf{X}_{j})$ is the quasi-likelihood for the $j$-piece, $n_1,\ldots,n_{m+1}$ are the segment lengths, $d_j$ is the dimension of $\boldsymbol{\theta}_j$, and $\zeta_{j,1},\ldots,\zeta_{j,c_{j}}$ are integer-valued parameters that specify the order of the model for the $j$-th segment.

Given the local change-point estimates $\hat{\mathcal{J}}^{(1)}$, the refined change-points can be estimated by
\begin{equation*}
(\hat{m}^{(2)},\hat{\mathcal{J}}^{(2)})=\mathop{\arg\min}_{m=|\mathcal{J}|,\mathcal{J}\subseteq\hat{\mathcal{J}}^{(1)}} \textup{MDL}(m,\mathcal{J})\,.
\end{equation*}
Note that searching for the best subset in $\hat{\mathcal{J}}^{(1)}$ is much more computationally efficient than 
the traditional approach of optimizing the MDL over all possible change-point positions. 

\item[]{\bf \large Third Step: Final change-point estimates and confidence intervals}  \\
Define the extended local-window and the corresponding observations for the $j$-th estimated change-point $\hat{\tau}_j^{(2)} \in \hat{\mathcal{J}}^{(2)}$ by
\begin{equation*}
E_j(h)=\{\hat{\tau}_j^{(2)}-2h+1,\dots,\hat{\tau}_j^{(2)}+2h\} \textup{ and } X_{E_j(h)}=(X_{\hat{\tau}_j^{(2)}-2h+1},\dots,X_{\hat{\tau}_j^{(2)}+2h})\,,
\end{equation*}
respectively. Note that this construction ensures that each true change-point is within $(\frac{1}{4},\frac{3}{4})$ of each extended local window $E_j(h)$
with probability approaching 1. Let $L_j(\tau,\boldsymbol{\theta}_1,\boldsymbol{\theta}_2)=\sum_{t=\hat{\tau}_j^{(2)}-2h+1}^{\tau}l_t(\boldsymbol{\theta}_1)+\sum_{t=\tau+1}^{\hat{\tau}_j^{(2)}+2h}l_t(\boldsymbol{\theta}_2)$. 
For $j=1,\dots,\hat{m}^{(2)}$, define the final estimate as
	\begin{equation*}
	\hat{\tau}_j^{(3)}=\mathop{\arg\max}_{\tau\in(\hat{\tau}_j^{(2)}-h,\hat{\tau}_j^{(2)}+h]}L_j(\tau,\hat{\boldsymbol{\theta}}_j,\hat{\boldsymbol{\theta}}_{j+1})\,,
	\end{equation*}
	where $\hat{\boldsymbol{\theta}}_j=\hat{\boldsymbol{\theta}}_j(\tau)=\mathop{\arg\max}_{\boldsymbol{\theta}_1}\sum_{t=\hat{\tau}_j^{(2)}-2h+1}^{\tau}l_t(\boldsymbol{\theta}_1)$, and $\hat{\boldsymbol{\theta}}_{j+1}$ is defined analogously.  
Then, apply the procedures specified in Section \ref{sec:cp.dist} to obtain a confidence interval 
around each final change-point estimate $\hat{\tau}_j^{(3)}$. 
\end{itemize}

In the first scanning step, given $h$, the computational complexity for evaluating $S_h(t)$ for each $t$ is of order $O(h)$ since the size of $W_j(h)$ is $2h$. In the second step, minimizing MDL over $\hat{m}^{(1)}$ change-point candidates requires a computational complexity of $O((\hat{m}^{(1)})^2n)$ when using the optimal partition (OP) method of \cite{jackson2005algorithm}. Finally, as the computation is restricted to the extended local windows in the third step, the computational complexity is $O(\hat{m}^{(2)}h^2)$. As both $\hat{m}^{(1)}$ and $\hat{m}^{(2)}$ are finite, the total computational complexity in the complete procedure to find the final change-point estimates is $O(nh+h^2)$. As will be shown in Section \ref{sec:asy.prop}, the condition $h=d (\log n)^3$ is needed for some positive number $d$. When $h$ is as small as the order of $O((\log n)^3)$, the computation for change-point detection can be completed in $O(n (\log n)^3)$ steps. Therefore, the complete three-step GLRSM requires the computational complexity of $O\big(n (\log n)^3\big)$, which is lower than the order of $O(n^2)$ using a dynamic programing algorithm.

The GLRSM procedure includes one tuning parameter, the window radius $h$. As will be shown in Section \ref{sec:asy.prop}, it is theoretically crucial to choose an $h>d (\log n)^3$ for Theorem \ref{thm1} to hold, where $d$ is an unknown constant. If $h$ is of an order larger than $O((\log n)^3)$, e.g. $h=d_2(\log n)^4$, then the GLRSM is consistent for any choice of $d_2$.
Based on our empirical studies, it is found that $d_2=1/25$ and $h\geq100$
usually yields favorable results for various models and sample sizes. Thus, we suggest using $\max(100, (\log n)^4/25)$ as a rule-of-thumb choice of $h$. 

\subsection{Asymptotic properties}\label{sec:asy.prop}
In this subsection we investigate the asymptotic properties of the GLRSM procedure. Specifically, we show the consistency
of the estimated number of change-points and the coverage accuracy of the confidence intervals.
We introduce the following assumptions.

\begin{assumption} \label{asp1}
For any two consecutive segments $\mathbf{X}_{j}=\{X_{\tau_{j-1}+1},\dots,X_{\tau_j}\}$ and $\mathbf{X}_{j+1}=\{X_{\tau_{j}+1},\dots,X_{\tau_{j+1}}\}$, the expectation of the conditional likelihood function $E[l_k(\boldsymbol{\theta})]$ has a unique maximizer at $\boldsymbol{\theta}=\boldsymbol{\theta}^0_j$ for $k \in \{\tau_{j-1}+1, \ldots, \tau_{j}\}$ and at $\boldsymbol{\theta}=\boldsymbol{\theta}^0_{j+1}$ for $k \in \{\tau_{j}+1, \ldots, \tau_{j+1}\}$, with $\boldsymbol{\theta}^0_j \neq \boldsymbol{\theta}^0_{j+1}$. Moreover, $E[l_k(\boldsymbol{\theta}^0_{j+1})] < E[l_k(\boldsymbol{\theta}^0_{j})]$ for $k \in \{\tau_{j-1}+1, \ldots, \tau_{j}\}$ and $E[l_k(\boldsymbol{\theta}^0_{j})] < E[l_k(\boldsymbol{\theta}^0_{j+1})]$ for $k \in \{\tau_{j}+1, \ldots, \tau_{j+1}\}$.
\end{assumption}

\begin{assumption} \label{asp2}
	Within any segment, $l_{k}(\boldsymbol\theta)$ is a measurable and continuous function with respect to $\{X_{t}\}$, and is almost surely twice continuously differentiable with respect to $\boldsymbol\theta$.
\end{assumption}

\begin{assumption} \label{asp3}
	Let $Y_{k}({\boldsymbol{\theta}})=l_k({\boldsymbol{\theta}})-\mathbb{E}\left[l_k({\boldsymbol{\theta}})\right]$. For all $\boldsymbol{\theta\in\boldsymbol{\boldsymbol{\Theta}}}$, there exists a $K>0$ such that ${\rm{E}}(e^{\left|Y_{k}({\boldsymbol{\theta}})\right|})\leq K$ 
for all $k \in \mathbb{N}$.
\end{assumption}

\begin{assumption} \label{asp4}
	For all $\boldsymbol{\theta}_j\in\boldsymbol{\Theta}_j$, there exists an integrable function $G(\boldsymbol{X}_t)$ such that $\mathbb{E}(G(\boldsymbol{X}_t))<\infty$ and $|l_t(\boldsymbol{\theta}_j)|\leq G(\boldsymbol{X}_t)$.
\end{assumption}

Theorem \ref{thm1} below asserts that all change-points can be identified in an $h$-neighborhood of $\hat{\mathcal{J}}^{(1)}$. 
\begin{theorem} \label{thm1}
	Let the set of true change-points be $\mathcal{J}_0=(\tau_1^0,\dots,\tau_{m_0}^0)$ and the set of local change-point estimates be $\hat{\mathcal{J}}^{(1)}=\{\hat{\tau}_1^{(1)},\hat{\tau}_2^{(1)},\dots,\hat{\tau}_{\hat{m}^{(1)}}^{(1)}\}$, where $\hat{m}^{(1)}=|\hat{\mathcal{J}}^{(1)}|$. Suppose Assumptions \ref{asp1}-\ref{asp4} hold, $2h<n\epsilon_\lambda$ and $\epsilon_\lambda>c$ for some $c>0$, then there exists some $d>0$ such that for $h\geq d(\log n)^3$,
	\begin{equation*}
		\mathbb{P}\bigg(\max_{\tau\in\mathcal{J}_0}\min_{k=1,\dots,\hat{m}^{(1)}}|\tau-\hat{\tau}_k^{(1)}|<h\bigg)\to 1\,.
	\end{equation*}
\end{theorem}

Note that as the minimum distance between change-points is $n\epsilon_\lambda=O(n)$, the true number of change-points $m_0$ is finite. However, there is no guarantee that $\hat{m}^{(1)}$ equals the number of change-points. That is, the number of change-points may be overestimated. Nevertheless, the following theorem shows that the MDL model selection approach 
yields the consistency of the number and positions of the change-points.
\begin{theorem} \label{thm2}
	Under the setting in Theorem \ref{thm1} with $\epsilon_\lambda>c$ for some $c>0$, we have $\hat{m}^{(2)}\overset{p}{\to}m_0$. In addition, given that $\hat{m}^{(2)}=m_0$, we have
	\begin{equation*}
	\mathbb{P}\left(\max_{j=1,\dots,m_0}|\hat{\tau}^{(2)}-\tau_j^0|<h\right)\to 1\,.
	\end{equation*}
\end{theorem}

Since $h>d(\log n)^3\to\infty$, Theorem \ref{thm2} implies that $\max_{j=1,\dots,m_0}|\hat{\tau}_j^{(2)}-\tau_j^0|=O_p(h)$, which is not optimal compared with the typical rate of $O_p(1)$. Nevertheless, the interval $[\hat{\tau}_j-h+1,\hat{\tau}_j+h]$ covers the true change-point $\tau_j^0$ with probability approaching 1.  This allows the extended local-window to yield consistent final estimates and confidence intervals.
\begin{theorem} \label{thm3}
	Assume that the conditions of Theorem \ref{thm2} hold, $3h<n\epsilon_\lambda$, and Assumption \ref{asp4} holds. Then, we have
	\begin{equation*}
	\hat{\tau}_j^{(3)}-\tau_j^0\overset{d}{\to}\mathop{\arg\max}_{\tau \in \mathbb{Z}} W_{j,\tau}\,,
	\end{equation*}
	where $W_{j,\tau}$ is a double-sided random walk defined as
	\begin{equation*}
		W_{j,\tau}=\begin{cases}
		\sum_{t=\tau_j^0+1}^{\tau_j^0+\tau}(l_t(\boldsymbol{\theta}_j^0)-l_t(\boldsymbol{\theta}_{j+1}^0))\,, &\tau>0\,,\\
		0\,, & \tau=0\,, \\
		\sum_{t=\tau_j^0+\tau+1}^{\tau_j^0}(l_t(\boldsymbol{\theta}_{j+1}^0)-l_t(\boldsymbol{\theta}_{j}^0))\,, &\tau<0\,.
		\end{cases}
	\end{equation*}
	In particular, $\hat{\tau}_j^{(3)}-\tau_j^0=O_p(1)$. 
\end{theorem}

As the minimum distance between change-points is much larger than the window radius $h$, i.e., $n\epsilon_\lambda/h \rightarrow \infty$, the distances between the extended local windows $E_j(h)$s diverge to $\infty$. Under some weak dependence conditions on the processes, the CIs constructed are asymptotic independent. Thus, a Bonferroni-type argument implies that an approximate $1-\alpha$ simultaneous confidence interval for all the $\hat{m}^{(2)}$ change points can be constructed by a collection of $(1-\alpha)^{1/\hat{m}^{(2)}}$ CIs for each of a set of $\hat{m}^{(2)}$ change points.

\section{Simulation studies\label{sec:simulation}}

In this section, we examine the finite sample performance of GLRSM via extensive simulation studies. 
In Sections \ref{appLRSM}, we study the performance of GLRSM on various models. In Section \ref{(GARCHperformance)}, we compare GLRSM to other methods in the literature, including \cite{fryzlewicz2014multiple}, in terms of the performance of change-point detection.

\subsection{Performance of GLRSM\label{appLRSM}}

In this section, we apply GLRSM to multiple change-points problems in different piecewise stationary time series models, including AR, ARMA, and GARCH models. 
Also, the rule-of-thumb choice of $h=\max(100,(\log n)^4/25)$ for the window radius is employed. 

We applied GLRSM on the piecewise stationary processes generated from multiple change-point Models (C) to (G) listed below. 
Realizations of these models are given in Figure \ref{fig:sample}. In each model, $\epsilon_{t}\overset{i.i.d.}\sim N(0,1)$.
The estimation results are summarized in Table \ref{tab:simuresult}.

\begin{figure}[!ht] 
	\centering
	\includegraphics[totalheight=5in]{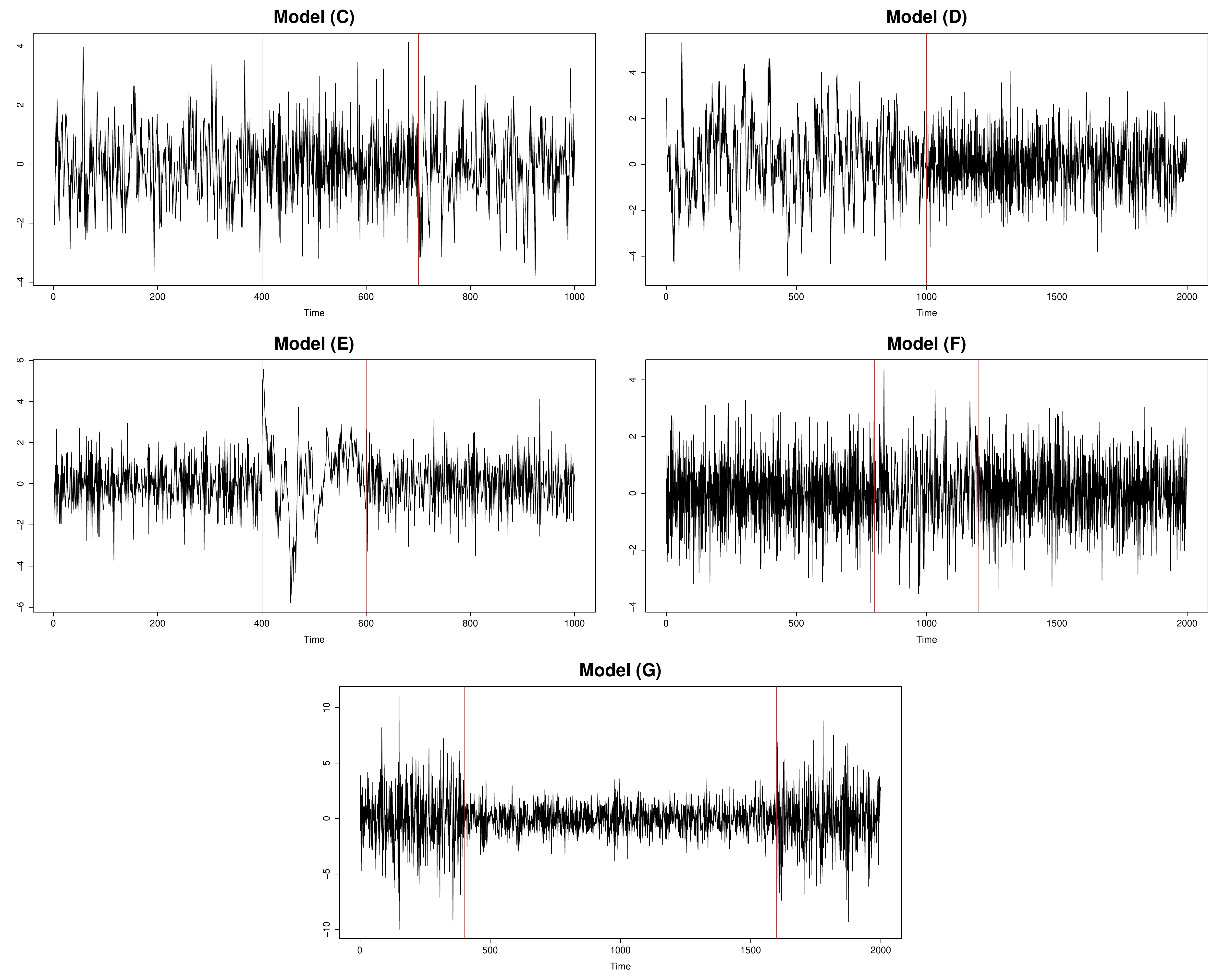}
	\caption{Realizations from Models (C) to (G). Vertical lines represent change-point positions} \label{fig:sample}
\end{figure}

\begin{table}[!ht]
	\scriptsize
	\centering
	\begin{tabular}{*{7}{c}}
		\toprule
		Model  & $\tau^0$ & Median & Mean & 90\% Range & Mean of 90\% CI & Coverage Prob.\\\hdashline
		\multirow{2}*{(C)}
		& 400   & 400 & 400.78 & [391, 411] & [393.62, 407.94] & 83.2\%\\
		& 700   & 700 & 700.09 & [691, 710] & [694.18, 705.99] & 83.1\%\\\midrule
		\multirow{2}*{(D)}
		& 1000   & 1000 &  999.35 & [991, 1007] & [994.68, 1004.02] & 81.7\%\\
		& 1500   & 1500 & 1499.66 & [1487, 1511] & [1490.83, 1508.50] & 84.6\%\\\midrule
		\multirow{2}*{(E)}
		& 400 & 400 & 400.56 & [394, 407] & [397.81, 403.31] & 71.1\%\\
		& 600 & 600 & 599.38 & [593, 605] & [595.37, 603.39] & 85.2\%\\\midrule
		\multirow{2}*{(F)}
		& 800 & 800 & 800.01 & [793, 806] & [795.08, 804.95] & 85.8\%\\
		& 1200 & 1200 & 1199.18 & [1188, 1209] & [1191.55, 1206.81] & 84.7\%\\\midrule
		\multirow{2}*{(G)}
		& 400 & 399 & 398.92 & [391, 406] & [394.76, 403.09] & 78.6\%\\
		& 1600 & 1600 & 1599.72 & [1593, 1609] & [1593.51, 1605.92] & 82.5\%\\\bottomrule
	\end{tabular}
	\caption{\footnotesize True value ($\tau^0$), median, mean and the range of the middle 90\% (90\% Range) of the final estimates, average of the end-points (mean of 90\% CI) and coverage probability (Coverage Prob.) of 90\% confidence intervals for Models (C) to (G). 
The number of replications is 1000.
}
	\label{tab:simuresult}
\end{table}

\noindent\textit{(C). Piecewise stationary AR(1) processes}\\
\begin{equation*}
X_{t}=\begin{cases}
0.4X_{t-1}+\epsilon_t\,,\ &\textup{if}\ 1\leq t\leq 400\,,\\
-0.6X_{t-1}+\epsilon_t\,,\ &\textup{if}\ 401\leq t\leq 700\,,\\
0.5X_{t-1}+\epsilon_t\,,\ &\textup{if}\ 701\leq t\leq 1000\,.
\end{cases}
\end{equation*}

\noindent\textit{(D). Piecewise stationary AR(2) processes}\\
\begin{equation*}
X_{t}=\begin{cases}
0.7X_{t-1}+0.1X_{t-2}+\epsilon_t\,,\ &\textup{if}\ 1\leq t\leq 1000\,,\\
-0.4X_{t-1}\epsilon_t\,,\ &\textup{if}\ 1001\leq t\leq 1500\,,\\
0.5X_{t-1}-0.2X_{t-2}+\epsilon_t\,,\ &\textup{if}\ 1501\leq t\leq 2000 \,.
\end{cases}
\end{equation*}

\noindent\textit{(E). Piecewise stationary ARMA(1,1) processes}\\
\begin{equation*}
	X_{t}=\begin{cases}
		-0.8X_{t-1}+\epsilon_t+0.5\epsilon_{t-1}\,,\ &\textup{if}\ 1\leq t\leq 400\,,\\
		0.9X_{t-1}+\epsilon_t \,, \ &\textup{if}\ 401\leq t\leq 600 \,,\\
		0.1X_{t-1}+\epsilon_t-0.5\epsilon_{t-1} \,,\ &\textup{if}\ 601\leq t\leq 1000 \,.
	\end{cases}
\end{equation*}

\noindent\textit{(F). Piecewise stationary ARMA processes up to order 2}\\
\begin{equation*}
	X_{t}=\begin{cases}
		-0.6X_{t-1}-0.2X_{t-2}+\epsilon_t\,,\ &\textup{if}\ 1\leq t\leq800 \,, \\
		0.4X_{t-1}+\epsilon_t+0.3\epsilon_{t-1}\,,\ &\textup{if}\ 801\leq t\leq1200 \,,\\
		\epsilon_t-0.3\epsilon_{t-1}-0.2\epsilon_{t-2}\,,\ &\textup{if}\ 1201\leq t\leq 2000 \,.
	\end{cases}
\end{equation*}

\noindent\textit{(G). Piecewise stationary GARCH processes}\\
\begin{equation*} 
X_t=\epsilon_t\sigma_t\,, \hspace{1cm} \sigma_t^2=
	\begin{cases}
		3+0.1 X_{t-1}^2+0.5 \sigma_{t-1}^2\,, \ &\textup{if}\ 1\leq t\leq 400\,,\\
		0.5+0.1 X_{t-1}^2+0.5 \sigma_{t-1}^2\,, \ &\textup{if}\ 401\leq t\leq 1600\,,\\
		0.8+0.1 X_{t-1}^2+0.8 \sigma_{t-1}^2\,, \ &\textup{if}\ 1601\leq t\leq 2000\,.
	\end{cases}
\end{equation*}

When performing GLRSM on ARMA($p$,$q$) models such as Model (E) and (F), in consideration of the computation complexity and the robustness of AR models, we  conduct the first two steps using piecewise AR models. After obtaining $\hat{\tau}_j^{(2)}$ from the second step, we use the ARMA models to conduct the third step and compute the final estimate $\hat{\tau}_j^{(3)}$, and construct the corresponding confidence intervals for each change-point estimate. 
Note that the consistency of the change-point estimator does not require that the true models be specified since the asymptotic results in Theorem \ref{thm1} and \ref{thm2} for the change-point estimates are valid when Assumptions \ref{asp1}-\ref{asp4} hold.

From Table \ref{tab:simuresult}, it can be seen that the final change-points estimates are accurate and the coverage probability is quite close to the nominal level of $90\%$. Taking into account all the simulation results, we conclude that GLRSM has a stable performance in most scenarios.

\subsection{GARCH change-point detection performance evaluation\label{(GARCHperformance)}}

In this section we compare the detection accuracy of GLRSM to other multiple change-point estimation procedures such as the BASTA-res in \cite{fryzlewicz2014multiple}. The codes for the BASTA-res are available at http://stats.lse.ac.uk/fryzlewicz/basta/basta.html.  
Following the same setting in \cite{fryzlewicz2014multiple}, we consider different GARCH(1,1) models with sample size $n=1000$ and one change-point occurring at time $t=500$, i.e.,
\begin{equation*}
X_t=\epsilon_t\sigma_t\,, \hspace{1cm} \sigma_t^2=
	\begin{cases}
		\omega_1+\alpha_1 X_{t-1}^2+\beta_1 \sigma_{t-1}^2\,, \ &\textup{if}\ 1\leq t\leq 500\,,\\
		\omega_2+\alpha_2 X_{t-1}^2+\beta_2 \sigma_{t-1}^2\,, \ &\textup{if}\ 501\leq t\leq 1000\,,
	\end{cases}
\end{equation*}
where $\epsilon_{t}\overset{i.i.d.}\sim N(0,1)$. The following ten change-point models are studied.

\vspace{0.5cm}

\begin{tabular}{|l|ccc|ccc|c|l|ccc|ccc|}
		\hline
		Model & $(\omega_1,\alpha_1,\beta_1)$ & $(\omega_2,\alpha_2,\beta_2)$ & & Model & $(\omega_1,\alpha_1,\beta_1)$ & $(\omega_2,\alpha_2,\beta_2)$ \\\hline
		(a)   & $(0.4,0.1,0.5)$ & $(0.4,0.1,0.5)$							  & & (b)   & $(0.1,0.1,0.8)$ & $(0.1,0.1,0.8)$ \\	
		(c)   & $(0.4,0.1,0.5)$ & $(0.4,0.1,0.6)$							  & & (d)   & $(0.4,0.1,0.5)$ & $(0.4,0.1,0.8)$ \\	
		(e)   & $(0.1,0.1,0.8)$ & $(0.1,0.1,0.7)$							  & & (f)   & $(0.1,0.1,0.8)$ & $(0.1,0.1,0.4)$ \\	
		(g)   & $(0.4,0.1,0.5)$ & $(0.5,0.1,0.5)$							  & & (h)   & $(0.4,0.1,0.5)$ & $(0.8,0.1,0.5)$ \\	
		(i)   & $(0.1,0.1,0.8)$ & $(0.3,0.1,0.8)$							  & & (j)   & $(0.1,0.1,0.8)$ & $(0.5,0.1,0.8)$ \\	\hline
\end{tabular}
\vspace{0.5cm}

Note that there is no change-point in Models (a) and (b). Since the BASTA-res does not produce confidence intervals for the change-points, only the proportion of correct number, mean, median, and the standard error of the estimated change-points are compared. The simulation results are summarized in 
Table \ref{tab:garchperf}.

\begin{table}[!ht]
	\footnotesize
	\centering
	\begin{tabular}{*{7}{c}}
		\toprule
		Method & Statistic & (a) & (b) & (c) & (d) & (e) \\\midrule
		\multirow{4}{1.8cm}{BASTA-res} & Accuracy rate & 94.0\% & 89.0\% & 34.8\% & 78.4\% & 71.4\% \\
		& Mean of est. & NA & NA & 519.99 & 502.70 & 493.04 \\
		& Median of est. & NA & NA & 503 & 501 & 499 \\
		& s.e. of est. & NA & NA & 148.43 & 24.16 & 91.99 \\\midrule
		\multirow{4}{1.8cm}{GLRSM} & Accuracy rate & 100\% & 100\% & 4.2\% & 95.2\% & 21.8\% \\
		& Mean of est. & NA & NA & 520.71 & 504.27 & 484.80 \\
		& Median of est. & NA & NA & 505 & 501 & 493 \\
		& s.e. of est. & NA & NA & 43.30 & 21.61 & 50.67 \\\midrule
		
		Method & Statistic & (f) & (g) & (h) & (i) & (j) \\\midrule
		\multirow{4}{1.8cm}{BASTA-res} & Accuracy rate & 77.8\% & 25.2\% & 89.0\% & 68.0\% & 70.4\% \\
		& Mean of est. & 499.34 & 490.85 & 497.93 & 500.07 & 502.66 \\
		& Median of est. & 501 & 495.5 & 500 & 501 & 501 \\
		& s.e. of est. & 15.82 & 150.26 & 75.84 & 47.53 & 20.13 \\\midrule
		\multirow{4}{1.8cm}{GLRSM} & Accuracy rate & 93.4\% & 0.4\% & 65.2\% & 69.6\% & 93.2\% \\
		& Mean of est. & 499.68 & 461.50 & 502.75 & 499.57 & 503.90 \\
		& Median of est. & 500 & 461.5 & 501 & 501 & 501 \\
		& s.e. of est. & 13.41 & 111.02 & 43.88 & 43.98 & 19.31 \\\bottomrule
	\end{tabular}
	\caption{\footnotesize Proportion of times that the correct number of change-points was detected, mean, median and standard error (s.e.) of the change-point estimates for different estimation methods for Models (a) to (j). For GLRSM, the scan window radius $h=\max(100, (\log n)^4/25)$ is used, where $n$ is the length of the series. The number of replications is 500.  
}
	\label{tab:garchperf}
\end{table}

From Table \ref{tab:garchperf}, when the structural change is small (Models (c), (e) and (g)), the performance of GLRSM is not as good as the BASTA-res method in \cite{fryzlewicz2014multiple}. One possible reason is that the short scanning window has less power to identify the change-point when the change is small. However, the GLRSM has higher credibility when no structural change actually occurs (Models (a) and (b)). Also, the GLRSM performs significantly better when the structural difference is large (Models (d), (f) and (j)). In all scenarios, the biases of the change-point estimates given by GLRSM and BASTA-res are similar. However, the standard error of the change-point estimates by GLRSM is always smaller than the standard error by BASTA-res. That is, compared with BASTA-res, the change-point estimates by GLRSM have higher precision.

\section{Conclusion\label{sec:conclusion}}

In this paper, we developed GLRSM as a computationally efficient procedure for multiple change-points inference in general time series models.
Compared with other methods that only yield point estimates, the GLRSM provides a confidence interval around each estimated change-point.
The results of the simulations and real data analyses indicated that the GLRSM performs well in change-points detection and has accurate coverage probabilities. In this paper, the use of likelihood ratio scan statistics is restricted to univariate time series models. Nevertheless, with proper modifications to the scanning procedures and large deviation bounds, the GLRSM could potentially be extended to multivariate settings.

\section{Proofs\label{sec:proof}}

First we state the following lemmas about the asymptotic properties of the likelihood ratio scan statistics, $S_h(t)$, which are required to prove Theorem \ref{thm1}.

\begin{lem} \label{lem1}
	For the j-th change-point $\tau^{0}_{j}$, the scan statistic $S_{h}(\tau^{0}_{j})$ satisfies
	\begin{equation*}
	\begin{aligned}
	S_{h}(\tau^{0}_{j})&=\frac{1}{h}L_{1h}(\tau^{0}_{j},\hat{\boldsymbol{\theta}}_{1})+\frac{1}{h}L_{2h}(\tau^{0}_{j},\hat{\boldsymbol{\theta}}_{2})-\frac{1}{h}L_{\cdot h}(\tau^{0}_{j},\hat{\boldsymbol{\theta}}_{1,2})\\
&\overset{p}{\to}\mathbb{E}\bigg(\log\frac{f_{k,\boldsymbol{\theta}_1}(\boldsymbol\theta_{1})}{f_{k,\boldsymbol{\theta}_1}(\boldsymbol\theta_{1,2})}\bigg)+\mathbb{E}\bigg(\log\frac{f_{k,\boldsymbol{\theta}_2}(\boldsymbol\theta_{2})}{f_{k,\boldsymbol{\theta}_2}(\boldsymbol\theta_{1,2})}\bigg) \triangleq g_{j}>0 \,,
	\end{aligned}
	\end{equation*}
	where
	\begin{equation*}
	\begin{aligned} \hat{\boldsymbol{\theta}}_{1,2}=\operatorname*{arg\,max}_{\boldsymbol{\theta\in\Theta}}\frac{1}{h}\bigg[\sum_{k=\tau^{0}_{j}-h+1}^{\tau^{0}_{j}}l_{k,\boldsymbol{\theta}_1}(\boldsymbol\theta)+\sum_{k=\tau^{0}_{j}+1}^{\tau^{0}_{j}+h}l_{k, \boldsymbol{\theta}_2}(\boldsymbol\theta)\bigg]\,,
	\end{aligned}
	\end{equation*}
	and
	\begin{equation*}
	\begin{aligned} 
	\boldsymbol{\theta}_{1,2}= \operatorname*{arg\,max}_{\boldsymbol{\theta\in\Theta}}\big[\mathbb{E}\big(l_{k,\boldsymbol{\theta}_1}(\boldsymbol\theta)\big)+\mathbb{E}\big(l_{k,\boldsymbol{\theta}_2}(\boldsymbol\theta)\big)\big]\,.
	\end{aligned}
	\end{equation*}
\end{lem}

\begin{lem} \label{lem2}
	For any $\epsilon>$0, there exists a positive integer $H$ such that for any $h>H$,
	\begin{equation*}
	\begin{aligned}
	\mathbb{P}(|S_{h}(t)|>\epsilon)\leq 6\exp(-\frac{1}{4} h^{1/3} \epsilon^{2/3})\,, 
	\end{aligned}
	\end{equation*}
	for all $t$ such that $W_t(h)$ does not contain any change-point.
\end{lem}

\begin{lem} \label{lem3}
	For any $\epsilon>0$, there exists a positive integer $H$ such that for any $h>H$,
	\begin{equation*}
	\begin{aligned}
	\mathbb{P}\big(|S_{h}(\tau^{0}_{j})-g_j|>\epsilon\big)\leq 22\exp(-\frac{1}{4} h^{1/3} \epsilon^{2/3})\,,
	\end{aligned}
	\end{equation*}
	for all $j=1,\dots,m_0$.
\end{lem}

\begin{proof}[\bf{Proof of Theorem \ref{thm1}:}]
\rm Let $A_t=$ \{some point in the $t$-th local-window is a local change-point estimate\} and let $A=\bigcap_{t\in\mathcal{J}_0}A_t$. The proof of Theorem \ref{thm1} is completed if we can prove \textbf{P}($A$)$\to 1$ as $n\to\infty$.

Let $\mathbb{Z}_n=\{1,2,\dots,n\}$. Define $\mathcal{E}=\mathbb{Z}_n\backslash\{\bigcup_{t\in\mathcal{J}_0}W_{t}(h)\}$ as the set of all points outside the $h$-neighborhood of the true change-points. One sufficient condition for the event $A$ to occur is that
\begin{equation} \label{thm1_1}
	\min_{t\in\mathcal{J}_0}S_h(t)>\max_{t\in\mathcal{E}}S_h(t)\,.
\end{equation}

Let $g=\frac{1}{2}\min_{j=1,\dots,m_0}(g_j)$ where $g_j$'s are defined in Lemma \ref{lem1}. Note from \eqref{thm1_1} that

\begin{equation} \label{thm1_2}
\mathbb{P}(A)\geq \mathbb{P}\bigg(\min_{t\in\mathcal{J}_0}S_h(t)>g>\max_{t\in\mathcal{E}}S_h(t)\bigg)\,.
\end{equation}

From Lemma \ref{lem3} and the definition of $g$, it can be shown that $\mathbb{P}\big(S_{h}(t)\leq g\big)\leq 22\exp(-\frac{1}{4} h^{1/3} g^{2/3})$, for all $t\in\mathcal{J}_0$. Thus,
\begin{equation}\label{thm1_3}
\begin{aligned}
\mathbb{P}\big(\min_{t\in\mathcal{J}_0}S_h(t)>g\big)&=1-\mathbb{P}\bigg(\bigcup_{t\in\mathcal{J}_0}\{S_h(t)\leq g\}\bigg)\geq 1-\sum_{t\in\mathcal{J}_0}\mathbb{P}(S_h(t)\leq g)\\
&\geq 1-22 m_0 \exp(-\frac{1}{4} h^{1/3} g^{2/3}) \to 1\,, 
\end{aligned}
\end{equation}
provided that $h=d (\log n)^3$ for some $d>0$ and $m_0=O(1)$. When $t\in\mathcal{E}$, all observations in $S_h(t)$ belong to one stationary piece. By Lemma \ref{lem2}, we have $\mathbb{P}(S_{h}(t)\geq g)<6\exp(-\frac{1}{4} h^{1/3} g^{2/3})$ for all $t\in\mathcal{E}$. Let $h=d (\log n)^3$ for some $d>64/g^2$, we have
\begin{equation}\label{thm1_4}
\begin{aligned}
\mathbb{P}\big(g>\max_{t\in\mathcal{E}}S_h(t)\big)&=1-\mathbb{P}\bigg(\bigcup_{t\in\mathcal{E}}\{S_h(t)\geq g\}\bigg)\geq 1-\sum_{j=1}^{m_0+1}(\tau_j^0-\tau_{j-1}^0)\mathbb{P}(S_h(t_j)\geq g)\\
&>1-6 n \exp(-\frac{1}{4} h^{1/3} g^{2/3})\to 1\,, 
\end{aligned}
\end{equation}
where $t_j\in(\tau_{j-1}^0+h, \tau_j^0-h)$. Combining \eqref{thm1_2}, \eqref{thm1_3}, and \eqref{thm1_4} yields $\mathbb{P}(A)\to 1$ when $h=d (\log n)^3$ for some $d>64/g^2$, completing the proof of Theorem \ref{thm1}.
\end{proof}

\begin{proof}[\bf{Proof of Theorem \ref{thm2}:}]
\rm The proof of this theorem is essentially the same as that of Theorem 2 in \cite{yau2015inference}, and hence is omitted.
\end{proof}

\begin{proof}[\bf{Proof of Theorem \ref{thm3}:} ]
\rm Based on Assumption \ref{asp4}, there exists a function $G(\boldsymbol{Z}_t)$ that dominates $|l_t(\boldsymbol{\theta}_j)|$ for all $\boldsymbol{\theta}_j\in\boldsymbol{\Theta}_j$. By using the uniform law of large number (ULLN) in \cite{jennrich1969asymptotic}, we have as $h\to\infty$, $\frac{1}{h}\sum_{t=1}^{h}l_t(\boldsymbol{\theta}_j)$ converges uniformly to $\mathbb{E}_{\boldsymbol{\theta}_j^0}(l_t(\boldsymbol{\theta}_j)), \forall \boldsymbol{\theta}_j\in\boldsymbol{\Theta}_j$. The remaining part of the proof is the same as that of Theorem 3 in \cite{yau2015inference}.

In practice, if the form of the underlying model is known, it is not difficult to find a dominating function $G(\boldsymbol{Z}_t)$ that satisfying Assumption \ref{asp4}. For example, if an AR($p$) model is used, we can take $G(Z_t,\dots,Z_{t-p})=|\frac{1}{2}\log(2\pi \sigma^2)|+(\sum_{j=0}^{p}Z_j^2+2\sum_{i\neq j}|Z_i Z_j|)/2\sigma^2$. 
\end{proof}
\begin{proof}[\bf{Proof of Lemma \ref{lem1}:}]
\rm Denote the conditional likelihood function at time $k$ given $\mathcal{F}_{k-1}$ with parameter $\boldsymbol{\theta}$ as $f(X_k,\boldsymbol{\theta})$.
As $h\to\infty$, we have 
\begin{eqnarray*}
\hat{\boldsymbol{\theta}}_{1,2}&=&\operatorname*{arg\,max}_{\boldsymbol{\theta\in\Theta}}\frac{1}{h}\big[\sum_{k=\tau^{0}_{j}-h+1}^{\tau^{0}_{j}}l_{1}(k,\boldsymbol\theta)+\sum_{k=\tau^{0}_{j}+1}^{\tau^{0}_{j}+h}l_{2}(k,\boldsymbol\theta)\big]\\
&=& \operatorname*{arg\,max}_{\boldsymbol{\theta\in\Theta}}\big[\mathbb{E}\big(l_{1}(k,\boldsymbol\theta)\big)+\mathbb{E}\big(l_2(k,\boldsymbol\theta)\big)+o_p(1)\big] 
\overset{p}{\to}
\boldsymbol{\theta}_{1,2}\,.
\end{eqnarray*}
Combining with $\hat{\boldsymbol{\theta}}_1\overset{p}{\to}\boldsymbol{\theta}_1$ and $\hat{\boldsymbol{\theta}}_2\overset{p}{\to}\boldsymbol{\theta}_2$, we have
\begin{align}
S_h(\tau_j^0)&=\frac{1}{h}\sum_{k=\tau^{0}_{j}-h+1}^{\tau^{0}_{j}}\log(f_{k,\boldsymbol{\theta_1}}(\hat{\boldsymbol\theta}_{1})+\frac{1}{h}\sum_{k=\tau^{0}_{j}+1}^{\tau^{0}_{j}+h}\log(f_{k,\boldsymbol{\theta}_2}(\hat{\boldsymbol\theta}_{2}))-\frac{1}{h}\sum_{k=\tau^{0}_{j}-h+1}^{\tau^{0}_{j}}\log(f_{k,\boldsymbol{\theta}_1}(\hat{\boldsymbol\theta}_{1,2}))\nonumber\\
&\quad-\frac{1}{h}\sum_{k=\tau^{0}_{j}+1}^{\tau^{0}_{j}+h}\log(f_{k,\boldsymbol{\theta}_2}(\hat{\boldsymbol\theta}_{1,2}))\nonumber\\
&=\frac{1}{h}\sum_{k=\tau^{0}_{j}-h+1}^{\tau^{0}_{j}}\log\frac{f_{k,\boldsymbol{\theta}_1}(\hat{\boldsymbol\theta}_{1})}{f_{k,\boldsymbol{\theta}_1}(\hat{\boldsymbol\theta}_{1,2})}+\frac{1}{h}\sum_{k=\tau^{0}_{j}+1}^{\tau^{0}_{j}+h}\log\frac{f_{k,\boldsymbol{\theta}_2}(\hat{\boldsymbol\theta}_{2})}{f_{k,\boldsymbol{\theta}_2}(\hat{\boldsymbol\theta}_{1,2})}\nonumber\\
&\overset{p}{\to}\mathbb{E}\bigg(\log\frac{f_{k,\boldsymbol{\theta}_1}(\boldsymbol\theta_{1})}{f_{k,\boldsymbol{\theta}_1}(\boldsymbol\theta_{1,2})}\bigg)+\mathbb{E}\bigg(\log\frac{f_{k,\boldsymbol{\theta}_2}(\boldsymbol\theta_{2})}{f_{k,\boldsymbol{\theta}_2}(\boldsymbol\theta_{1,2})}\bigg)=g_j\,. 
\label{two.E.gj}
\end{align}
For the first expectation in \eqref{two.E.gj}, 
the Jensen's Inequality implies that
\begin{equation*}
\mathbb{E}\bigg(\log\frac{f_{k,\boldsymbol{\theta}_1}(\boldsymbol\theta_{1})}{f_{k,\boldsymbol{\theta}_1}(\boldsymbol\theta_{1,2})}\bigg)=-\mathbb{E}\bigg(\log\frac{f_{k,\boldsymbol{\theta}_1}(\boldsymbol\theta_{1,2})}{f_{k,\boldsymbol{\theta}_1}(\boldsymbol\theta_1)}\bigg)\geq -\log\bigg[\mathbb{E}\bigg(\frac{f_{k,\boldsymbol{\theta}_1}(\boldsymbol\theta_{1,2})}{f_{k,\boldsymbol{\theta}_1}(\boldsymbol\theta_1)}\bigg)\bigg]\,.
\end{equation*}
Since the observations come from the segment with the true parameter $\boldsymbol\theta_1$, we thus have
\begin{align}
\log\bigg[\mathbb{E}\bigg(\frac{f_{k,\boldsymbol{\theta}_1}(\boldsymbol\theta_{1,2})}{f_{k,\boldsymbol{\theta}_1}(\boldsymbol\theta_1)}\bigg)\bigg]&=\log\bigg[\int_{-\infty}^{\infty}\frac{f_{k,\boldsymbol{\theta}_1}(\boldsymbol\theta_{1,2})}{f_{k,\boldsymbol{\theta}_1}(\boldsymbol\theta_1)}f_{k,\boldsymbol{\theta}_1}(\boldsymbol\theta_1)dx_k\bigg]\nonumber 
=\log 1=0\,.
\end{align}
Notice that since $\boldsymbol\theta_{1,2}$ depends on two distinct time series models specified by $\boldsymbol\theta_1$ and $\boldsymbol\theta_2$ respectively, then $\frac{f_{k,\boldsymbol{\theta}_1}(\boldsymbol\theta_{1,2})}{f_{k,\boldsymbol{\theta}_1}(\boldsymbol\theta_1)}$ does not satisfy the condition that $\frac{f_{k,\boldsymbol{\theta}_1}(\boldsymbol\theta_{1,2})}{f_{k,\boldsymbol{\theta}_1}(\boldsymbol\theta_1)}\to c$ a.s., where $c$ is a constant. Also, the log-function is not a linear function. Therefore, $\mathbb{E}\big(\log\frac{f_{k,\boldsymbol{\theta}_1}(\boldsymbol\theta_{1})}{f_{k,\boldsymbol{\theta}_1}(\boldsymbol\theta_{1,2})}\big)>0$, and the equality could not be achieved. Finally, using similar arguments for the second expectation in \eqref{two.E.gj}, we have $g_j>0$, and thus Lemma \ref{lem1} follows.
\end{proof}
\begin{proof}[\bf{Proof of Lemma \ref{lem2}:}]
\rm When there is no change-point in the scanning window $W_{t}(h)$, we use the notation $l_k(\boldsymbol{\theta})$ to represent $l_{k,\boldsymbol{\theta}_0}(\boldsymbol{\theta})$, since all data are from the segment specified by $\boldsymbol{\theta}_0$. Hence, we can write the scan statistic as
\begin{equation} \label{sht}
S_{h}(t)=\frac{1}{h}\sum_{k=t-h+1}^t[l_{k}(\hat{\boldsymbol\theta}_{1})-l_{k}(\boldsymbol\theta_{0})]+\frac{1}{h}\sum_{k=t+1}^{t+h}[l_{k}(\hat{\boldsymbol\theta}_{2})-l_{k}(\boldsymbol\theta_{0})]-\frac{1}{h}\sum_{k=t-h+1}^{t+h}[l_{k}(\hat{\boldsymbol\theta})-l_{k}(\boldsymbol\theta_{0})]\,,
\end{equation}
where $\hat{\boldsymbol\theta}_{1}$, $\hat{\boldsymbol\theta}_{2}$, and $\hat{\boldsymbol\theta}$ are the maximum likelihood estimators (MLE) of the parameter $\boldsymbol\theta$ in the left half, right half, and the entire scanning window, respectively.

We first decompose the third sum in \eqref{sht} into the following form,
\begin{align} \label{sht2}
	\frac{1}{h}\sum_{k=t-h+1}^{t+h}[l_{k}(\hat{\boldsymbol\theta})-l_{k}(\boldsymbol\theta_{0})] &= \frac{1}{h}\sum_{k=t-h+1}^{t+h}[l_{k}(\hat{\boldsymbol\theta})-{\rm{E}} (l_{k}(\hat{\boldsymbol\theta}))] - \frac{1}{h}\sum_{k=t-h+1}^{t+h}[l_{k}(\boldsymbol\theta_0)-{\rm{E}} (l_{k}(\boldsymbol\theta_0))] \nonumber\\
	& \quad + [{\rm{E}} (l_{k}(\hat{\boldsymbol\theta})) - {\rm{E}} (l_{k}(\boldsymbol\theta_0))] \,.
\end{align}
By Assumptions \ref{asp2}, \ref{asp3} and the compactness of the parameter space $\Theta$, 
there exists a positive constant $K^*$, such that ${\rm{E}} (e^{\left|Y_{k}({\boldsymbol{\theta}})\right|})\leq K^*$ for all $k \in \mathbb{N}$ and  
$\boldsymbol{\theta} \in \boldsymbol{\Theta}$. Then by Theorem 3.2 of \cite{lesigne2001large}, 
for any $\epsilon>0$, there exists a positive integer $H_0$ depending only on $K^*$ and $\epsilon$, 
such that for any $h>H_0$ and $\boldsymbol{\theta} \in \boldsymbol{\Theta}$, 
\begin{equation} \label{liklp}
\mathbb{P}\bigg(\bigg|\frac{1}{h}\sum_{k=t-h+1}^{t+h}[l_{k}(\boldsymbol\theta)-{\rm{E}} (l_{k}(\boldsymbol\theta))]\bigg|> \epsilon/3 \bigg)\leq e^{-\frac{1}{4} h^{1/3} \epsilon^{2/3}}\,,
\end{equation}
Moreover, the exponential moment condition ${\rm{E}} (e^{\left|Y_{k}({\boldsymbol{\theta}})\right|})\leq K^*$ for all $\boldsymbol{\theta} \in \boldsymbol{\Theta}$ also implies the uniform integrability of $\{l_{k}(\boldsymbol\theta): \boldsymbol\theta \in \boldsymbol\Theta\}$. Thus, together with $\hat{\boldsymbol\theta}\overset{p}{\to} \boldsymbol\theta_0$, we have that for any $\epsilon>0$, there exists a constant $H_1>0$ such that for any $h>H_1$,
\begin{equation}\label{expcov}
\left|{\rm{E}} (l_{k}(\hat{\boldsymbol\theta})) - {\rm{E}} (l_{k}(\boldsymbol\theta_0))\right| < \epsilon/3\,.
\end{equation}

As a result, applying \eqref{liklp} on both $\hat{\boldsymbol{\theta}}$ and $\boldsymbol{\theta}_0$ and \eqref{expcov} we have that for any $\epsilon>0$, there exists a positive integer $H^*$ such that for any $h>H^*$,
\begin{equation*}
\mathbb{P}\bigg(\bigg|\frac{1}{h}\sum_{k=t-h+1}^{t+h}[l_{k}(\hat{\boldsymbol\theta})-l_{k}(\boldsymbol\theta_{0})]\bigg|>\epsilon \bigg)\leq 2e^{-\frac{1}{4} h^{1/3} \epsilon^{2/3}}\,.
\end{equation*}
Similarly exponential inequalities  hold for the other two sums in {\eqref{sht}}, hence Lemma {\ref{lem2}} follows.
\end{proof}

\begin{proof}[\bf{Proof of Lemma \ref{lem3}:}]
\rm Using the notations in Lemma \ref{lem1}, we have
\begin{align}
S_h(\tau_j^0)-g_j&=\bigg[\frac{1}{h}\sum_{k=\tau_j^0-h+1}^{\tau_j^0}l_{k,\boldsymbol{\theta}_1}(\hat{\boldsymbol{\theta}}_1)-E(l_{k,\boldsymbol{\theta}_1}(\boldsymbol{\theta}_1))\bigg]+\bigg[\frac{1}{h}\sum_{k=\tau_j^0+1}^{\tau_j^0+h}l_{k,\boldsymbol{\theta}_2}(\hat{\boldsymbol{\theta}}_2)-E(l_{k,\boldsymbol{\theta}_2}(\boldsymbol{\theta}_2))\bigg]\nonumber\\
&\quad-\bigg[\frac{1}{h}\sum_{k=\tau_j^0-h+1}^{\tau_j^0}\big(l_{k,\boldsymbol{\theta}_1}(\hat{\boldsymbol{\theta}}_{1,2})+l_{k+h,\boldsymbol{\theta}_2}(\hat{\boldsymbol{\theta}}_{1,2})\big)-E\big(l_{k,\boldsymbol{\theta}_1}(\boldsymbol{\theta}_{1,2})\big)-E\big(l_{k+h,\boldsymbol{\theta}_2}(\boldsymbol{\theta}_{1,2})\big)
\bigg]\nonumber\\
&=\bigg[\frac{1}{h}\sum_{k=\tau_j^0-h+1}^{\tau_j^0}\big(l_{k,\boldsymbol{\theta}_1}(\hat{\boldsymbol{\theta}}_1)-(l_{k,\boldsymbol{\theta}_1}(\boldsymbol{\theta}_1)\big)\bigg]
+\bigg[\frac{1}{h}\sum_{k=\tau_j^0-h+1}^{\tau_j^0}l_{k,\boldsymbol{\theta}_1}(\boldsymbol{\theta}_1)-E(l_{k,\boldsymbol{\theta}_1}(\boldsymbol{\theta}_1))\bigg]\nonumber\\
&\quad+\bigg[\frac{1}{h}\sum_{k=\tau_j^0+1}^{\tau_j^0+h}\big(l_{k,\boldsymbol{\theta}_2}(\hat{\boldsymbol{\theta}}_2)-(l_{k,\boldsymbol{\theta}_2}(\boldsymbol{\theta}_2)\big)\bigg]
+\bigg[\frac{1}{h}\sum_{k=\tau_j^0+1}^{\tau_j^0+h}l_{k,\boldsymbol{\theta}_1}(\boldsymbol{\theta}_2)-E(l_{k,\boldsymbol{\theta}_2}(\boldsymbol{\theta}_2))\bigg]\nonumber\\
&\quad+\frac{1}{h}\sum_{k=\tau_j^0-h+1}^{\tau_j^0}\bigg(\big(l_{k,\boldsymbol{\theta}_1}(\hat{\boldsymbol{\theta}}_{1,2})+l_{k+h,\boldsymbol{\theta}_2}(\hat{\boldsymbol{\theta}}_{1,2})\big)-\big(l_{k,\boldsymbol{\theta}_1}(\boldsymbol{\theta}_{1,2})+l_{k+h,\boldsymbol{\theta}_2}(\boldsymbol{\theta}_{1,2})\big)\bigg)\nonumber\\
&\quad+\bigg[\frac{1}{h}\sum_{k=\tau_j^0-h+1}^{\tau_j^0}l_{k,\boldsymbol{\theta}_1}(\boldsymbol{\theta}_{1,2})-E\big(l_{k,\boldsymbol{\theta}_1}(\boldsymbol{\theta}_{1,2})\big)\bigg]+\bigg[\frac{1}{h}\sum_{k=\tau_j^0-h+1}^{\tau_j^0}l_{k,\boldsymbol{\theta}_2}(\boldsymbol{\theta}_{1,2})-E\big(l_{k+h,\boldsymbol{\theta}_2}(\boldsymbol{\theta}_{1,2})\big)\bigg]\,. \label{lem3expansion}
\end{align}
By Lemma \ref{lem2}, we have that for any $\epsilon>0$, there exists a constant $H_2>0$ such that for any $h>H_2$,
\begin{equation*}
\mathbb{P}\bigg(\bigg|\frac{1}{h}\sum_{k=\tau_j^0-h+1}^{\tau_j^0}\big(l_{k,\boldsymbol{\theta}_1}(\tau_j^0,\hat{\boldsymbol{\theta}}_1)-l_{k,\boldsymbol{\theta}_1}(\tau_j^0,\boldsymbol{\theta}_1)\big)\bigg|>\epsilon\bigg)\leq 6e^{-\frac{1}{4} h^{1/3} \epsilon^{2/3}}\,,
\end{equation*}
and
\begin{equation*}
\mathbb{P}\bigg(\bigg|\frac{1}{h}\sum_{k=\tau_j^0+1}^{\tau_j^0+h}\big(l_{k,\boldsymbol{\theta}_2}(\tau_j^0,\hat{\boldsymbol{\theta}}_2)-l_{k,\boldsymbol{\theta}_2}(\tau_j^0,\boldsymbol{\theta}_2)\big)\bigg|>\epsilon\bigg)\leq 6e^{-\frac{1}{4} h^{1/3} \epsilon^{2/3}}\,.
\end{equation*}

By the definition of $\boldsymbol{\theta_{1,2}}$, we have $\frac{\partial}{\partial \boldsymbol\theta}E\big(l_{k,\boldsymbol{\theta_1}}(\boldsymbol{\theta})+l_{k,\boldsymbol{\theta_2}}(\boldsymbol{\theta})\big)\big|_{\boldsymbol{\theta}=\boldsymbol{\theta_{1,2}}}=0$. Combining with Assumption \ref{asp2} gives $E\frac{\partial}{\partial \boldsymbol\theta}\big(l_{k,\boldsymbol{\theta_1}}(\boldsymbol{\theta_{1,2}})+l_{k,\boldsymbol{\theta_2}}(\boldsymbol{\theta_{1,2}})\big)=0$. Hence,  $\left\{\frac{\partial}{\partial \boldsymbol\theta}\big(l_{k,\boldsymbol{\theta_1}}(\boldsymbol{\theta}_{1,2})+l_{k+h,\boldsymbol{\theta_2}}(\boldsymbol{\theta}_{1,2})\big)\right\}$ is a martingale difference sequence. Using similar arguments in the proof of Lemma \ref{lem2}, we have that for any $\epsilon>0$, there exists a constant $H_3>0$ such that for any $h>H_3$,
\begin{equation*}
\mathbb{P}\bigg(\bigg|\frac{1}{h}\sum_{k=\tau_j^0-h+1}^{\tau_j^0}\big[\big(l_{k,\boldsymbol{\theta}_1}(\tau_j^0,\hat{\boldsymbol{\theta}}_{1,2})+l_{k+h,\boldsymbol{\theta}_2}(\tau_j^0,\hat{\boldsymbol{\theta}}_{1,2})\big)-\big(l_{k,\boldsymbol{\theta}_1}(\tau_j^0,\boldsymbol{\theta}_{1,2})+l_{k+h,\boldsymbol{\theta}_2}(\tau_j^0,\boldsymbol{\theta}_{1,2})\big)\big]\bigg|>\epsilon\bigg)\leq 6e^{-\frac{1}{4} h^{1/3} \epsilon^{2/3}}\,.
\end{equation*}
The remaining four terms in \eqref{lem3expansion} can be bounded by considering the large deviation probabilities of  $\frac{1}{h}\sum_{k=\tau_j^0-h+1}^{\tau_j^0}l_{k,\boldsymbol{\theta}_1}(\boldsymbol{\theta}_1)-E(l_{k,\boldsymbol{\theta}_1}(\boldsymbol{\theta}_1))$. 
For example, let $Y_{k}(\boldsymbol{\theta}_1)=l_{k,\boldsymbol{\theta}_1}(\boldsymbol{\theta}_1)-E(l_{k,\boldsymbol{\theta}_1}(\boldsymbol{\theta}_1))$. Note that $\{Y_{k}(\boldsymbol{\theta}_1)\}$ is a martingale difference sequence. By Assumption \ref{asp3} and Theorem 3.2 of \cite{lesigne2001large}, for any $\epsilon>0$, there exists a constant $H_4>0$ such that for any $h>H_4$,
\begin{equation*}
\mathbb{P}\bigg(\bigg|\frac{1}{h}\sum_{k=\tau_j^0-h+1}^{\tau_j^0}Y_k\bigg|\geq\epsilon\bigg)\leq e^{-\frac{1}{4} h^{1/3} \epsilon^{2/3}}\,.
\end{equation*}
Therefore, the large deviation probability for every term in \eqref{lem3expansion} has been obtained. The large deviation bound for $S_h(\tau_j^0)-g_j$ can thus be established and the proof is complete.
\end{proof}

\bibliographystyle{apalike}
\bibliography{literature}

\end{document}